\documentclass[12pt]{amsart}
\usepackage[all]{xy}
\usepackage[parfill]{parskip}

\usepackage{verbatim}
\usepackage{color}

\usepackage{amsmath, amscd, graphicx, latexsym, hyperref, times, calc}
\usepackage{color,graphicx}
\usepackage[abs]{overpic}
\usepackage{tikz}
\usepackage{tikz-cd}
\usetikzlibrary{arrows, patterns}

\oddsidemargin .25in \theoremstyle{plain}
\newtheorem{dummy}{anything}[section]
\newtheorem{theorem}[dummy]{Theorem}

\newtheorem{lemma}[dummy]{Lemma}

\newtheorem{proposition}[dummy]{Proposition}

\newtheorem{corollary}[dummy]{Corollary}

\theoremstyle{definition}
\newtheorem{definition}[dummy]{Definition}

\newtheorem{example}[dummy]{Example}

\newtheorem{remark}[dummy]{Remark}

\newtheorem*{acknowledgements}{Acknowledgements}

\begin{document}

\title{Contact $(+1)$-surgeries on rational homology 3-spheres}

\author{Fan Ding, Youlin Li and Zhongtao Wu}

\address{School of Mathematical Sciences and LMAM, Peking University, Beijing 100871, China}
\email{dingfan@math.pku.edu.cn}

\address{School of Mathematical Sciences, Shanghai Jiao Tong University, Shanghai 200240, China}
\email{liyoulin@sjtu.edu.cn}

\address{Department of Mathematics, The Chinese University of Hong Kong, Shatin, Hong Kong}
\email{ztwu@math.cuhk.edu.hk}

\subjclass[2000]{}

\begin{abstract}
In this paper, sufficient conditions for contact $(+1)$-surgeries along Legendrian knots in contact rational homology 3-spheres to have vanishing contact invariants or to be overtwisted are given. They can be applied to study contact $(\pm1)$-surgeries along Legendrian links in the standard contact 3-sphere. We also obtain a sufficient condition for contact $(+1)$-surgeries along Legendrian two-component links in the standard contact 3-sphere to be overtwisted via their front projections.
\end{abstract}

\maketitle

\section{Introduction}\label{sec: intro}

There is a dichotomy of contact structures on 3-manifolds: tight and overtwisted. Given a contact 3-manifold $(Y,\xi)$, it is a fundamental question to ask whether it is tight or overtwisted. In \cite{OSzContact}, Ozsv\'{a}th and Szab\'{o} introduced contact invariants $c(\xi)\in \widehat{HF}(-Y)$ and its image $c^{+}(\xi)\in HF^{+}(-Y)$, and proved that if $(Y,\xi)$ is overtwisted then $c(\xi)$ vanishes. Moreover, Ghiggini proved that if $(Y,\xi)$ is strongly symplectically fillable then $c^{+}(\xi)$, and hence $c(\xi)$,  are non-trivial \cite[Theorem 2.13]{g1}. So it is crucial to determine whether the contact invariant is trivial or not. In \cite{dg}, the first author and Geiges proved that any closed connected contact 3-manifold can be obtained by contact surgery on the standard contact 3-sphere $(S^{3},\xi_{st})$ along a Legendrian link $\mathbb{L}_{1}\cup \mathbb{L}_{2}$ with coefficients $+1$ for each component of $\mathbb{L}_1$ and $-1$ for each component of $\mathbb{L}_2$.  This leads us to study the tightness and contact invariant of $(Y,\xi)$ through its contact $(\pm1)$-surgery diagram along a Legendrian link in  $(S^{3},\xi_{st})$. If $\mathbb{L}_1$ is empty, then $(Y,\xi)$ is Stein fillable, and $c(\xi)$ is nontrivial, and hence $(Y,\xi)$ is tight. So we consider the case that  $\mathbb{L}_1$ is non-empty, namely there are contact $(+1)$-surgeries involved in the surgery. In \cite{dlw}, the authors studied contact $(+1)$-surgeries along Legendrian two-component links in $(S^{3},\xi_{st})$.

In many situations, a problem related to contact $(\pm1)$-surgery along a Legendrian link in  $(S^{3},\xi_{st})$ can be reduced to a problem related to contact $(+1)$-surgery along  a Legendrian link in a contact rational homology 3-sphere. This may occur, for example, when contact $(-1)$-surgery along the sublink $\mathbb{L}_2$ yields a contact rational homology 3-sphere.

For a Legendrian knot $L$ in the standard contact 3-sphere $(S^3, \xi_{st})$, whether or not the result of contact $\frac{p}{q}$-surgery along $L$ has non-vanishing contact invariant has been well studied by Lisca and Stipsicz \cite{ls3, ls1}, Golla \cite{gol}, and Mark and Tosun \cite{mt}, etc. Recall that the contact invariant is natural under the cobordism induced by contact $(+1)$-surgery \cite{OSzContact}. Thus, if $(Y,\xi)$ is a contact 3-manifold whose contact invariant vanishes, then the result of contact $(+1)$-surgery along any Legendrian knot in $(Y,\xi)$ has vanishing contact invariant as well. By Theorem 1.2 in \cite{w} and Proposition 8 in \cite{dg1}, the result of contact $(+1)$-surgery along any Legendrian knot in an overtwisted closed connected contact 3-manifold $(Y,\xi)$ is overtwisted. In this paper, we are mainly concerned with the contact invariant and overtwistedness of the result of contact $(+1)$-surgery along a Legendrian knot in a contact rational homology 3-sphere.

The last two authors introduced an invariant $\tau^{\ast}_{c(\xi)}(Y, K)$ for a rationally null-homologous knot $K$ in a contact 3-manifold $(Y,\xi)$ with non-vanishing contact invariant $c(\xi)$ \cite{lw}, and proved that this invariant gives an upper bound for the sum of the rational Thurston-Bennequin invariant and the absolute value of the rational rotation number of all Legendrian knots isotopic to $K$, i.e. $$tb_{\mathbb{Q}}(L)+|rot_{\mathbb{Q}}(L)|\leq 2\tau^{\ast}_{c(\xi)}(Y, K)-1,$$
where $L$ is a Legendrian knot in $(Y,\xi)$ isotopic to $K$. This is a generalization of the inequalities appeared in \cite{b}, \cite{e}, \cite{p}, \cite{h}, etc. We give a sufficient condition for the result of contact $(+1)$-surgery having vanishing contact invariant.
Let $(Y_{+1}(L), \xi_{+1}(L))$ denote the result of contact $(+1)$-surgery on $(Y, \xi)$ along $L$.

\begin{theorem}\label{Theorem:Main1}
Suppose $K$ is a knot in a rational homology 3-sphere $Y$, and $\xi$ is a contact structure on $Y$ with nontrivial contact invariant $c(\xi)\in \widehat{HF}(-Y)$. Let $L$ be a Legendrian knot in $(Y,\xi)$ isotopic to $K$.  Then the  contact invariant $c(\xi_{+1}(L))$ vanishes if
$$tb_{\mathbb{Q}}(L)+|rot_{\mathbb{Q}}(L)|<2\tau^{\ast}_{c(\xi)}(Y,K)-1.$$
\end{theorem}

 Let $(S^3(\mathbb{L}_{1}^{+}\cup \mathbb{L}_{2}^{-}),\xi_{st}(\mathbb{L}_1^+\cup\mathbb{L}_2^-))$ denote the contact 3-manifold obtained by contact surgery on $(S^{3},\xi_{st})$ along a Legendrian link $\mathbb{L}_{1}\cup \mathbb{L}_{2}$ with coefficients $+1$ for each component of $\mathbb{L}_1$ and $-1$ for each component of $\mathbb{L}_2$.  \"{O}zba\u{g}ci showed in \cite{Ozb} that if some component of $\mathbb{L}_{1}$ contains an isolated stabilized arc which does not tangle with any other component of $\mathbb{L}_{1}\cup \mathbb{L}_{2}$, then $(S^3(\mathbb{L}_1^+\cup\mathbb{L}_2^-),\xi_{st}(\mathbb{L}_{1}^{+}\cup \mathbb{L}_{2}^{-}))$ is overtwisted. In fact, thanks to Theorem 1.2 in \cite{w}, the condition in \"{O}zba\u{g}ci's result can be slightly relaxed  to be that some component of $\mathbb{L}_{1}$ contains an isolated stabilized arc which does not tangle with any component of $\mathbb{L}_{2}$. Applying Theorem~\ref{Theorem:Main1} we obtain a result similar to that of \"{O}zba\u{g}ci. Here we consider isolated Legendrian connected summands. See Figure~\ref{figure:connsum}. We refer the reader to \cite{eh} for Legendrian connected sum.

\begin{proposition}\label{Pro0}
Let $\mathbb{L}_{1}\cup \mathbb{L}_{2}\subset (S^{3},\xi_{st})$ be an oriented Legendrian link.  If the contact 3-manifold  $S^3(\mathbb{L}_{2}^{-})$ is a rational homology 3-sphere, and there exists a front projection of $\mathbb{L}_{1}\cup \mathbb{L}_{2}$ such that a component $L_1$ of $\mathbb{L}_{1}$ contains an isolated connected summand $L_3$ which does not tangle with $\mathbb{L}_{2}$ and satisfies
$$tb(L_3)+|rot(L_3)|<2\tau(L_3)-1,$$ then the contact invariant $c(\xi_{st}(\mathbb{L}_{1}^{+}\cup \mathbb{L}_{2}^{-}))$ vanishes.
\end{proposition}

\begin{figure}[htb]
\begin{overpic}
{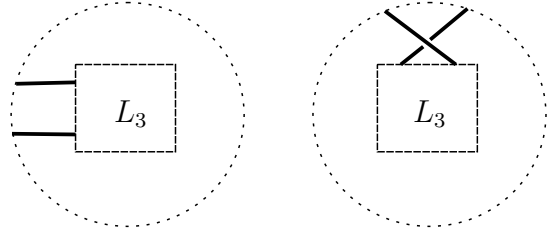}
\put(40, 40){$L_3$}
\put(153, 40){$L_3$}
\end{overpic}
\caption{Isolated Legendrian connected sum.}
\label{figure:connsum}
\end{figure}

On the other hand, we remark that the second part of \cite[Proposition 1.4]{ls1} can be generalized to Legendrian knots in contact L-spaces.

\begin{proposition}\label{Prop:tb}
Suppose $(Y,\xi)$ is a contact L-space, and $L$ is a Legendrian knot in $(Y,\xi)$. If $tb_{\mathbb{Q}}(L)<-1$, then the contact invariant $c^{+}(\xi_{+1}(L))$ vanishes.
\end{proposition}

As mentioned earlier, this result may be used to study certain contact $(\pm1)$-surgeries along Legendrian links in $(S^3, \xi_{st})$.

\begin{corollary}\label{cor0}
Let $\mathbb{L}=L_{1}\cup L_{2}$ be a Legendrian two-component link in $(S^{3}, \xi_{st})$ whose two components have linking number $l$. Suppose that $L_2$ is a Legendrian L-space knot and $l^2>2g(L_2)(tb(L_1)+1)$, where $g(L_2)$ denotes the genus of $L_2$.  Then the contact invariant $c^{+}(\xi_{st}(\mathbb{L}^{+}))$ vanishes.
\end{corollary}

Now we deal with the overtwistedness of contact $(+1)$-surgeries. Among other things,  Conway \cite{c} and Onaran \cite{o} obtained sufficient conditions for the overtwistedness of contact $(+1)$-surgeries along Legendrian null-homologous knots. Here we generalize Conway's results to Legendrian knots in contact rational homology 3-spheres. Likewise, it is useful for determining the overtwistedness of contact $(\pm1)$-surgery along Legendrian links in $(S^3,\xi_{st})$.

\begin{theorem}\label{Theorem:Main3}
 Let $L$ be a Legendrian knot in a contact rational homology 3-sphere $(Y, \xi)$. Let $q$ be the order of $[L]$ in $H_{1}(Y;\mathbb{Z})$, and $\chi(F)$ be the Euler characteristic of a rational Seifert surface $F$ for $L$. \newline
 (1) If  $$tb_{\mathbb{Q}}(L)<-1~~ \text{and} ~~tb_{\mathbb{Q}}(L)-|rot_{\mathbb{Q}}(L)|<\frac{\chi(F)}{q},$$
 then $(Y_{+1}(L), \xi_{+1}(L))$ is overtwisted. \newline
 (2) If $$tb_{\mathbb{Q}}(L)+|rot_{\mathbb{Q}}(L)|<\frac{\chi(F)}{q}-2,$$
 then the result of any positive contact surgery along $L$ is overtwisted.
\end{theorem}

In \cite[Theorem 1.6 and Corollary 6.4]{dlw}, the authors obtained sufficient conditions for the result of contact $(+1)$-surgery along a Legendrian two-component link in $(S^3,\xi_{st})$ to be overtwisted via some specific configurations in the front projection. The following theorem is an improvement of \cite[Corollary 6.4]{dlw}.

\begin{theorem}\label{Theorem:Main2}
Suppose the front projection of a Legendrian two-component link
$\mathbb{L}=L_{1}\cup L_{2}$ in the standard contact 3-sphere $(S^3, \xi_{st})$ contains a configuration
exhibited in Figure~\ref{figure:ot}, then contact $(+1)$-surgery on $(S^3, \xi_{st})$ along $\mathbb{L}$ yields an
overtwisted contact 3-manifold.
\end{theorem}

\begin{figure}[htb]
\begin{overpic}
{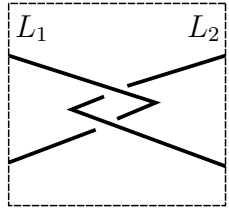}
\put(3, 65){$L_1$}
\put(68, 65){$L_{2}$}
\end{overpic}
\caption{A configuration in the front projection of a Legendrian two-component
link $\mathbb{L}$.}
\label{figure:ot}
\end{figure}

\medskip
\begin{acknowledgements}
The authors would like to thank Roger Casals for sharing us his alternative proof of Theorem~\ref{Theorem:Main2}. The first author was partially supported by Grant No. 11371033 of the National Natural Science Foundation of China. The second author was partially supported by Grant No. 11871332 of the National Natural Science Foundation of China. The third author was partially supported by grants from the Research Grants Council of the Hong Kong Special Administrative Region, China (Project No. 14300018).
\end{acknowledgements}

\section{knots in rational homology 3-spheres} \label{sec: invariant}

Suppose $Y$ is an oriented rational homology 3-sphere and $K$ is an oriented knot in $Y$. Suppose the order of $[K]$ in $H_{1}(Y;\mathbb{Z})$ is $q$. According to \cite{be}, we define a \emph{rational Seifert surface}  for $K$ to be a smooth map $j:F\rightarrow Y$ from a connected compact oriented surface $F$ to $Y$ that is an embedding from the interior of $F$ into $Y\setminus K$, and a $q$-fold cover from its boundary $\partial F$ to $K$.  Denote by $N(K)$ a closed tubular neighborhood of $K$ in $Y$, and by $\mu\subset\partial N(K)$ a meridian of $K$. We can assume that $j(F)
\cap\partial N(K)$ is composed of $c$ parallel oriented simple closed curves, each of which has homology $\nu\in H_{1}(\partial N(K); \mathbb{Z})$. Then we can choose a canonical longitude $\lambda_{can}$ such that $\nu=t\lambda_{can}+r\mu$, where homology classes of $\lambda_{can}$ and $\mu$ are also denoted by $\lambda_{can}$ and $\mu$, respectively, $t$ and $r$ are coprime integers, and $0\leq r<t$ (cf. \cite[Section 2.6]{r}). Certainly we have $ct=q$.

\subsection{Filtrations.}

Let $(\Sigma, \boldsymbol{\alpha}, \boldsymbol{\beta}, w, z)$ be a doubly pointed Heegaard diagram of $K$ in $Y$. Then the set of relative Spin$^{c}$-structures determines a filtration of the chain complex $\widehat{CF}(Y)$ via a map
$$\underline{\mathfrak{s}}_{w,z}:\mathbb{T}_{\alpha}\cap\mathbb{T}_{\beta}\rightarrow \underline{\text{Spin}^{c}}(Y,K).$$
Each relative Spin$^{c}$ structure $\underline{\mathfrak{s}}$ for $(Y,K)$ corresponds to a Spin$^{c}$ structure $\mathfrak{s}$ on $Y$ via a natural map $G_{Y,K}:\underline{\text{Spin}^{c}(}Y,K)\rightarrow \text{Spin}^{c}(Y)$ (see \cite[Section 2.2]{Rsurgery}).

Fix a rational Seifert surface $F$ for $K$.   As in \cite{hl}, the Alexander grading of a relative Spin$^{c}$-structure $\underline{\mathfrak{s}} \in \underline{\text{Spin}^{c}(}Y,K)$ is defined by

$$A(\underline{\mathfrak{s}})=\frac{1}{2q}(\langle  c_{1}(\underline{\mathfrak{s}}), [F] \rangle+q).$$

Moreover, the Alexander grading of an intersection point $x\in \mathbb{T}_{\alpha}\cap\mathbb{T}_{\beta}$  is defined by
$$A(x)=\frac{1}{2q}(\langle c_{1}(\underline{\mathfrak{s}}_{w,z}(x)), [F]\rangle+q).$$

In general, the Alexander grading $A(x)$ is a rational number.  Nonetheless, observe that for any two relative Spin$^c$ structures $\underline{\mathfrak{s}}_1, \underline{\mathfrak{s}}_2 \in G_{Y,K}^{-1} (\mathfrak{s})$ of a fixed $\mathfrak{s}$, we have $\underline{\mathfrak{s}}_2-\underline{\mathfrak{s}}_1=l \, \text{PD}[\mu]$ for some integer $l$, and $A(\underline{\mathfrak{s}}_2)-A(\underline{\mathfrak{s}}_1)=l$.

\subsection{Rational $\tau$-invariants and rational $\nu$-invariants}

For a fixed $\underline{\mathfrak{s}}\in\underline{\text{Spin}^{c}}(Y,K)$, let $C_{\underline{\mathfrak{s}}}$ be the $\mathbb{Z}\oplus\mathbb{Z}$-filtered chain complex $CFK^{\infty}(Y,K,\underline{\mathfrak{s}})$ (see \cite[Section 3]{Rsurgery}).

Let $$\iota_{k}:C_{\underline{\mathfrak{s}}}\{i=0, j\leq k\}\rightarrow C_{\underline{\mathfrak{s}}}\{i=0\}$$  be the inclusion map, where $k\in \mathbb{Z}$. It induces a homomorphism between the homologies $$\iota_{k\ast}:H_{\ast}(C_{\underline{\mathfrak{s}}}\{i=0, j\leq k\})\rightarrow \widehat{HF}(Y, \mathfrak{s}),$$
where $\mathfrak{s}=G_{Y,K}(\underline{\mathfrak{s}})$. Let  $$v_{k}:C_{\underline{\mathfrak{s}}}\{\max(i,j-k)=0\}\rightarrow C_{\underline{\mathfrak{s}}}\{i=0\}$$  be the composition of $\iota_{k}$ and the quotient map from $C_{\underline{\mathfrak{s}}}\{\max(i,j-k)=0\}$ to $C_{\underline{\mathfrak{s}}}\{i=0, j\leq k\}$. It induces a homomorphism between the homologies $$v_{k\ast}:H_{\ast}(C_{\underline{\mathfrak{s}}}\{\max(i,j-k)=0\})\rightarrow \widehat{HF}(Y, \mathfrak{s}).$$

Next we recall the definition of  rational $\tau$ invariant \cite{h}.

\begin{definition}\label{def:tau}
For any $[x]\neq0\in \widehat{HF}(Y, \mathfrak{s})$, define
$$\tau_{[x]}(Y,K)=\text{min}\{A(\underline{\mathfrak{s}})+k|[x]\in \text{Im}(\iota_{k\ast})\}.$$
\end{definition}

Then we introduce the definition of rational $\nu$ invariant in the same manner as Hom-Levine-Lidman did in the integral homology sphere case (\cite{hll}).

\begin{definition}\label{def:nu}
For any $[x]\neq0\in \widehat{HF}(Y, \mathfrak{s})$, define
$$\nu_{[x]}(Y,K)=\text{min}\{A(\underline{\mathfrak{s}})+k|[x]\in \text{Im}(v_{k\ast})\}.$$
\end{definition}

\begin{lemma}\label{lemma:nutau}
$\nu_{[x]}(Y,K)=\tau_{[x]}(Y,K)$ or $\tau_{[x]}(Y,K)+1$.

\end{lemma}

\begin{proof}
Due to \cite[Proposition 24]{h} and \cite[Proposition 2.5]{lw}, the proof is similar to the case where $Y=S^3$ \cite[Equation 34]{Rsurgery}, and is straightforward.
\end{proof}

Consider the orientation reversal $-Y$ of $Y$.  We have a pairing
$$\langle-,-\rangle:\widehat{CF}(-Y,\mathfrak{s})\otimes\widehat{CF}(Y,\mathfrak{s})\rightarrow \mathbb{Z}/2\mathbb{Z},$$
given by
$$\langle x,y\rangle=
\begin{cases}
1& \text{if}~~   x=y,\\
0& \text{otherwise}.
\end{cases}$$
It descends to a pairing
$$\langle-,-\rangle:\widehat{HF}(-Y,\mathfrak{s})\otimes\widehat{HF}(Y,\mathfrak{s})\rightarrow \mathbb{Z}/2\mathbb{Z}.$$

\begin{definition}\label{def:tau*}
For any $[y]\neq0\in \widehat{HF}(-Y, \mathfrak{s})$, define
$$\tau^{\ast}_{[y]}(Y, K)=\text{min}\{A(\underline{\mathfrak{s}})+k|\exists\alpha\in \text{Im}(\iota_{k\ast}), \text{such that} \langle[y], \alpha\rangle\neq 0\}.$$
\end{definition}

\begin{proposition} \cite[Proposition 2.3]{lw}\label{prop: dual}
Let $[y]\neq0\in\widehat{HF}(-Y,\mathfrak{s})$. Then
$$\tau_{[y]}(-Y, K)=-\tau^{\ast}_{[y]}(Y, K).$$
\end{proposition}

\begin{proposition} \cite[Proposition 2.4]{lw}\label{prop: additivity0}
 For any  $[y_{i}]\neq0\in\widehat{HF}(-Y_{i}, \mathfrak{s}_{i})$, $i=1,2$, we have
$$\tau^{\ast}_{[y_1]\otimes[y_2]}(Y_{1}\sharp Y_{2}, K_{1}\sharp K_{2})=\tau^{\ast}_{[y_1]}(Y_{1}, K_{1})+\tau^{\ast}_{[y_2]}(Y_{2}, K_{2}).$$
\end{proposition}

Suppose $(Y_1, \xi_1)$ and $(Y_2, \xi_2)$ are two contact rational homology 3-spheres with nonvanishing contact invariants $c(\xi_1)$ and $c(\xi_2)$, then the contact invariant of $(Y_1\sharp Y_2, \xi_1\sharp\xi_2)$ is $c(\xi_1)\otimes c(\xi_2)\in \widehat{HF}(-Y_1)\otimes \widehat{HF}(-Y_2)$. See for example \cite[Page 105]{h}. As a corollary of Proposition~\ref{prop: additivity0}, we have the following proposition.

\begin{proposition}\label{prop: additivity}
Let $K_1\subset (Y_1, \xi_1)$ and $K_2\subset (Y_2, \xi_2)$ be two smooth knots. Then
$$\tau^{\ast}_{c(\xi_1)\otimes c(\xi_2)}(Y_{1}\sharp Y_{2}, K_{1}\sharp K_{2})=\tau^{\ast}_{c(\xi_1)}(Y_{1}, K_{1})+\tau^{\ast}_{c(\xi_2)}(Y_{2}, K_{2}).$$
\end{proposition}

\subsection{Mapping cone for Morse surgery along knots in rational homology 3-spheres} \label{sec: mappingcone}

For $p\in\mathbb{Z}$, let $Y_p(K)$ denote the 3-manifold obtained by performing Dehn surgery along $K$ in $Y$ with coefficient $p$ with respect to the canonical longitude $\lambda_{can}$. The meridian $\mu$ is isotopic to a core circle of the
glued-in solid torus. Let $K_p$ denote this core circle, with the orientation inherited from $-\mu$ (see \cite[Section 1.1]{hl}). The sets $\underline{\text{Spin}^c}(Y,K)$ and
$\underline{\text{Spin}^c}(Y_p(K),K_p)$ are naturally identified and the fibers of the map $$G_{Y_p(K),K_p}:\underline{\text{Spin}^c}(Y_p(K),K_p)\to \text{Spin}^c(Y_p(K))$$ are the orbits of
the action of $PD[\lambda]$, where $\lambda=\lambda_{can}+p\mu$. Let $X_p(K)$ be the 4-manifold obtained by attaching a 4-dimensional 2-handle $H$ to $Y\times I$ along $K\times\{ 1\}$
with coefficient $p$ with respect to $\lambda_{can}$. Then $\partial X_p(K)=(-Y)\sqcup Y_p(K)$. Let $C$ denote the core disk of the attached 2-handle in $X_p(K)$ with $\partial
C=K\times \{1\}$. For a rational Seifert surface $j:F\to Y$ for $K$, $F_{qC}=(j(F)\times\{ 1\})\cup (-qC)$ represents a homology class $[F_{qC}]$ in $H_2(X_p(K);\mathbb{Z})$. Given a Spin$^c$ structure $\mathfrak{s}$ on $Y$, we can extend $\mathfrak{s}$ to a Spin$^c$ structure $\mathfrak{t}$ on $X_p(K)$. All Spin$^c$ structures $\mathfrak{t}$ on $X_p(K)$ with $\mathfrak{t}|_Y=\mathfrak{s}$ are distinguished by $\left\langle c_1(\mathfrak{t}),[F_{qC}]\right\rangle$. For each $\underline{\mathfrak{s}}\in\underline{\text{Spin}^c}(Y,K)$
with $G_{Y,K}(\underline{\mathfrak{s}})=\mathfrak{s}$, there is a unique Spin$^c$ structure $\mathfrak{t}$ on $X_p(K)$ such that $\mathfrak{t}|_Y=\mathfrak{s}$ and
$$\left\langle c_{1}(\mathfrak{t}),[F_{qC}]\right\rangle +pq-cr=2qA(\underline{\mathfrak{s}})$$ (see \cite[Theorem 4.2]{r}).

For $\underline{\mathfrak{s}}\in\underline{\text{Spin}^c}(Y,K)$, let $A_{\underline{\mathfrak{s}}}=C_{\underline{\mathfrak{s}}}\{\max(i,j)=0\}$ and $B_{\underline{\mathfrak{s}}}=C_{\underline{\mathfrak{s}}}\{i=0\}$.
There are two natural projection maps
$$v_{\underline{\mathfrak{s}}}:A_{\underline{\mathfrak{s}}}\rightarrow B_{\underline{\mathfrak{s}}}, \;\; h_{\underline{\mathfrak{s}}}:A_{\underline{\mathfrak{s}}}\rightarrow B_{\underline{\mathfrak{s}}+PD[\lambda]}.$$
Define
$$\Phi : \bigoplus\limits_{\underline{\mathfrak{s}}} A_{\underline{\mathfrak{s}}} \rightarrow \bigoplus\limits_{\underline{\mathfrak{s}}} B_{\underline{\mathfrak{s}}}, \;\; (\underline{\mathfrak{s}}, a) \mapsto\left(\underline{\mathfrak{s}}, v_{\underline{\mathfrak{s}}}(a)\right)+\left(\underline{\mathfrak{s}}+PD[\lambda], h_{\underline{\mathfrak{s}}}(a)\right),$$
which is often written in the following form

$$\xymatrix{
\cdots\ar[drr] & & A_{\underline{\mathfrak{s}}-PD[\lambda]}\ar[d]^(0.6){v_{\underline{\mathfrak{s}}-PD[\lambda]}}\ar[drr]^(0.5){h_{\underline{\mathfrak{s}}-PD[\lambda]}} & & A_{\underline{\mathfrak{s}}}\ar[d]^(0.6){v_{\underline{\mathfrak{s}}}}\ar[drr]^(0.5){h_{\underline{\mathfrak{s}}}} & & A_{\underline{\mathfrak{s}}+PD[\lambda]}\ar[d]^(0.6){v_{\underline{\mathfrak{s}}+PD[\lambda]}}\ar[drr]^(0.5){h_{\underline{\mathfrak{s}}+PD[\lambda]}}& &\cdots\\
\cdots& &B_{\underline{\mathfrak{s}}-PD[\lambda]}& & B_{\underline{\mathfrak{s}}}& &B_{\underline{\mathfrak{s}}+PD[\lambda]}& &\cdots
}$$

Note that the mapping cone  of $\Phi$ splits over equivalence classes of relative Spin$^{c}$ structures, where $\underline{\mathfrak{s}}_{1}$ and $\underline{\mathfrak{s}}_{2}$ are equivalent if $\underline{\mathfrak{s}}_{2}=\underline{\mathfrak{s}}_{1}+mPD[\lambda]$ for some integer $m$.

\begin{theorem} \cite[Theorem 6.1]{Rsurgery}\label{iso}
Denote $\widehat{\mathbb{X}}_{[\underline{\mathfrak{s}}]}$ the summand of the cone of $\Phi$ corresponding to the equivalence class of $\underline{\mathfrak{s}}$. Then there exists a quasi-isomorphism of the complexes $$\Psi: \widehat{\mathbb{X}}_{[\underline{\mathfrak{s}}]}\cong \widehat{CF}(Y_p(K), G_{Y_{p}(K), K_{p}}(\underline{\mathfrak{s}})).$$
\end{theorem}

\begin{theorem} \cite[Theorem 4.2]{r}\label{rao}
 Suppose $\mathfrak{s}\in\text{Spin}^c(Y)$ and $\mathfrak{t}\in \text{Spin}^c(X_p(K))$ extends $\mathfrak{s}$. Then the map
 $$\widehat{CF}(Y,\mathfrak{s})\to\widehat{CF}(Y_p(K),\mathfrak{t}|_{Y_p(K)})$$ induced by $\mathfrak{t}$ corresponds via $\Psi$ to the inclusion of $B_{\underline{\mathfrak{s}}}$ in $\widehat{\mathbb{X}}_{[\underline{\mathfrak{s}}]}$, where $\underline{\mathfrak{s}}\in \underline{\text{Spin}^c}(Y,K)$ is determined by  $G_{Y,K}(\underline{\mathfrak{s}})=\mathfrak{s}$ and $$\left\langle c_{1}\left(\mathfrak{t}\right),[F_{qC}]\right\rangle +pq-cr=2qA(\underline{\mathfrak{s}}).$$
\end{theorem}

\section{Contact (+1)-surgeries on rational homology 3-spheres with vanishing contact invariants} \label{sec: legendrian}

 For a rationally null-homologous oriented Legendrian knot $L$ in a contact 3-manifold $(Y, \xi)$, one can define the rational Thurston-Bennequin invariant $tb_{\mathbb{Q}}(L)$ and the rational rotation number $rot_{\mathbb{Q}}(L)$. We refer the reader to \cite{be} for more details.

Suppose $Y$ is an oriented rational homology 3-sphere and $\xi$ is a contact structure on $Y$. Let $K$ be an oriented knot in $Y$. Suppose the order of $[K]$ in $H_1(Y;\mathbb{Z})$ is $q$. Let $L$ be an oriented Legendrian knot in $(Y,\xi)$ isotopic to $K$. Let $F$ be a rational Seifert surface for $L$.  Suppose the longitude of $L$ determined by the contact framing is $\lambda_{c}=\lambda_{can}+(p-1)\mu$ for some integer $p$. We use the notation in Section 2 with $K$ replaced by $L$. Performing contact $(+1)$-surgery along $L$, we obtain a contact 3-manifold $(Y_{+1}(L),\xi_{+1}(L))$. This contact $(+1)$-surgery induces a cobordism $X_p(L)$, also denoted by $W$, from  $Y$ to  $Y_{+1}(L)$. Notice that $Y_p(L)=Y_{+1}(L)$.

The contact structure $\xi$ (respectively, $\xi_{+1}(L)$) defines an almost complex structure $J$ (on $W$) along $Y$ (respectively, $Y_{+1}(L)$) by requiring $\xi$ (respectively, $\xi_{+1}(L)$) to be $J$-invariant and $J$ to map the inward (respectively, outward) normal along $Y$ (respectively, $Y_{+1}(L)$) to a vector in $Y$ (respectively, $Y_{+1}(L)$) positively transverse to $\xi$ (respectively, $\xi_{+1}(L)$). This $J$ can be extended to the closure of the complement of a 4-disk $D_H\subset int(H)\subset W$ such that $d_3(\xi_H)=\frac{1}{2}$, where $\xi_H$ denotes the plane field on $\partial D_H$ induced by $J$ (see \cite[Section 3]{dgs}).
The Spin$^c$ structure on $W$ induced by $J$ is denoted by $\mathfrak{t}_1$.

Mimicking the proof of \cite[Proposition 5.2]{dgs}, we have

\begin{lemma}\label{lemma: rot}
$\langle c_{1}(\mathfrak{t}_1), [F_{qC}]\rangle=q\cdot rot_{\mathbb{Q}}(L)$. \hfill $\Box$
\end{lemma}

\begin{lemma}\label{lemma: tb}
$pq-cr=q\cdot (tb_{\mathbb{Q}}(L)+1)$.
\end{lemma}

\begin{proof}
Recall that the longitude of $L$ determined by the contact framing is $\lambda_{c}=\lambda_{can}+(p-1)\mu$.  By definition in Baker-Etnyre \cite{be}, $tb_{\mathbb{Q}}(L)$ is the rational linking number of $L$ and $\lambda_{c}$. So
$$tb_{\mathbb{Q}}(L)+1=lk_{\mathbb{Q}}(L, \lambda_{c})+1=\frac{1}{q} [F]\bullet \lambda_{c}+1$$
$$=\frac{1}{q} (q\lambda_{can}+cr\mu)\bullet (\lambda_{can}+(p-1)\mu)+1
=\frac{1}{q}(pq-cr),$$
where the second intersection product is on $\partial(\overline{Y\setminus N(L)})$.
\end{proof}

By Baldwin \cite[Theorem 1.2]{bal} (see also Mark-Tosun \cite[Theorem 3.1]{mt}),
there exists a Spin$^{c}$ structure $\mathfrak{t}_2$ on $-W$ such that the homomorphism $$F_{-W,\mathfrak{t}_2}:\widehat{HF}(-Y)\rightarrow \widehat{HF}(-Y_{+1}(L))$$ satisfies
$$F_{-W,\mathfrak{t}_2}(c(\xi))=c(\xi_{+1}(L)).$$ The following lemma is similar to \cite[Corollary 3.6]{mt}.

\begin{lemma}\label{mt}
The Spin$^{c}$ structure $\mathfrak{t}_2$ satisfies $$\langle c_{1}(\mathfrak{t}_2), [F_{qC}]\rangle=\pm\langle c_{1}(\mathfrak{t}_1), [F_{qC}]\rangle=\pm q\cdot rot_{\mathbb{Q}}(L).$$
\end{lemma}

\begin{proof}

First, assume that $pq-cr=0$, which is equivalent to $r=p=0$. In this case, the map $H^2(X_p(L);\mathbb{Q})\to H^2(Y_{+1}(L);\mathbb{Q})$ induced by inclusion
is an isomorphism. Combining this with the fact that $\mathfrak{t}_1,\mathfrak{t}_2$ induce the same Spin$^c$ structure on $Y_{+1}(L)$, i.e. the Spin$^c$ structure induced by
$\xi_{+1}(L)$, we have $$\langle c_{1}(\mathfrak{t}_2), [F_{qC}]\rangle =\langle c_{1}(\mathfrak{t}_1), [F_{qC}]\rangle=q\cdot rot_{\mathbb{Q}}(L).$$

Now suppose $pq-cr\neq 0$. In this case, $Y_{+1}(L)$ is a rational homology 3-sphere and $c_1(\xi_{+1}(L))$ is torsion. By the degree shift formula in Heegaard Floer
homology, we have
\begin{equation} \label{c_1^2(t_2)}
\frac{1}{4}(c_1^2(\mathfrak{t}_2)_{-W}-3\sigma (-W)-2\chi (-W))=-d_3(\xi_{+1}(L))+d_3(\xi).
\end{equation}

Let $W'$ denote $W-int(D_H)$.
Since $\partial W'=Y_{+1}(L)\sqcup (-Y)\sqcup (-\partial D_H)$ and the almost complex structure $J$ induces the plane fields $\xi_{+1}(L),\xi$ and $\xi_H$ on $Y_{+1}(L),Y$ and
$\partial D_H$, respectively, we have
$$d_3(\xi_{+1}(L))-d_3(\xi)-d_3(\xi_H)=\frac{1}{4}(c_1^2(J)_{W'}-3\sigma (W')-2\chi (W'))$$
$$=\frac{1}{4}(c_1^2(\mathfrak{t}_1)_{W}-3\sigma (W)-2\chi (W))+\frac{1}{2}.$$
Note that $\chi(W)=1$ and $d_3(\xi_H)=\frac{1}{2}$. Thus
$$d_3(\xi_{+1}(L))-d_3(\xi)-\frac{1}{2}=\frac{1}{4}(c_1^2(\mathfrak{t}_1)_{W}-3\sigma (W)).$$
With (\ref{c_1^2(t_2)}), we get $c_{1}^{2}(\mathfrak{t}_2)_{-W}=-c_{1}^{2}(\mathfrak{t}_1)_{W}$. Since $H_2(W;\mathbb{Q})$ is generated by $[F_{qC}]$ and
$[F_{qC}]^2=q(pq-cr)\neq 0$, we obtain
$$\langle c_{1}(\mathfrak{t}_2), [F_{qC}]\rangle =\pm\langle c_{1}(\mathfrak{t}_1), [F_{qC}]\rangle=\pm q\cdot rot_{\mathbb{Q}}(L).$$
\end{proof}

For $\underline{\mathfrak{s}}\in \underline{\text{Spin}^c}(-Y,L)$, by Theorem~\ref{iso}, there exists a quasi-isomorphism of the complexes $$\Psi:\widehat{\mathbb{X}}_{[\underline{\mathfrak{s}}]}\cong \widehat{CF}(-Y_p(L), G_{-Y_{p}(L), -L_{p}}(\underline{\mathfrak{s}})).$$
The isomorphism
$$H_*(\widehat{\mathbb{X}}_{[\underline{\mathfrak{s}}]})\cong \widehat{HF}(-Y_p(L), G_{-Y_{p}(L), -L_{p}}(\underline{\mathfrak{s}}))$$
induced by $\Psi$ is denoted by $\Psi_*$.
Let $\mathfrak{s}_{\xi}$ (respectively, $\mathfrak{s}_{\xi_{+1}(L)}$) denote the Spin$^c$ structure on $Y$ (respectively, $Y_{+1}(L)$) induced by $\xi$ (respectively, $\xi_{+1}(L)$).
By Theorem~\ref{rao}, the map $$\widehat{CF}(-Y,\mathfrak{s}_{\xi})\to\widehat{CF}(-Y_p(L),\mathfrak{s}_{\xi_{+1}(L)})$$ induced by $\mathfrak{t}_2$ corresponds via $\Psi$ to the
inclusion of $B_{\underline{\mathfrak{s}}}$ in $\widehat{\mathbb{X}}_{[\underline{\mathfrak{s}}]}$, where $\underline{\mathfrak{s}}\in \underline{\text{Spin}^c}(-Y,L)$ satisfies $G_{-Y,L}(\underline{\mathfrak{s}})=\mathfrak{s}_{\xi}$ and $$\left\langle c_{1}\left(\mathfrak{t}_2\right),[F_{qC}]\right\rangle -pq+cr=2qA(\underline{\mathfrak{s}}).$$
Applying Lemmata~\ref{lemma: tb} and \ref{mt}, we have

\begin{corollary}\label{cor:cone}
Via $\Psi_*$, the contact invariant $c(\xi_{+1}(L))\in \widehat{HF}(-Y_{p}(L))$ is the image of $c(\xi)$ under the homomorphism of homologies induced by inclusion  $$B_{\underline{\mathfrak{s}}}\hookrightarrow \widehat{\mathbb{X}}_{[\underline{\mathfrak{s}}]},$$ where $\underline{\mathfrak{s}}\in \underline{\text{Spin}^c}(-Y,L)$ satisfies
$G_{-Y,L}(\underline{\mathfrak{s}})=\mathfrak{s}_{\xi}$ and $$ 2A(\underline{\mathfrak{s}})=-tb_{\mathbb{Q}}(L)\pm rot_{\mathbb{Q}}(L)-1.$$
\end{corollary}

\begin{lemma}\label{lemma: im}
Suppose $c(\xi)\neq 0$. For $\underline{\mathfrak{s}}$ in the above corollary,
$c(\xi)\not\in \text{Im}(v_{\underline{\mathfrak{s}}\ast})$ if and only if both $$\nu_{c(\xi)}(-Y, K)=-\tau^{\ast}_{c(\xi)}(Y,K)+1 ~~\text{and}~~ \tau^{\ast}_{c(\xi)}(Y,K)=\frac{1}{2}(tb_{\mathbb{Q}}(L)\mp rot_{\mathbb{Q}}(L)+1).$$
\end{lemma}

\begin{proof}
By the definition of $\nu_{c(\xi)}(-Y, K)$, the contact invariant $c(\xi)\not\in \text{Im}(v_{\underline{\mathfrak{s}}\ast})$ if and only if $A(\underline{\mathfrak{s}})<\nu_{c(\xi)}(-Y, K)$. By Lemma~\ref{lemma:nutau} and Proposition~\ref{prop: dual},  $\nu_{c(\xi)}(-Y, K)$ equals either $-\tau^{\ast}_{c(\xi)}(Y, K)$ or $-\tau^{\ast}_{c(\xi)}(Y, K)+1$. On the other hand, it follows from \cite[Theorem 1.1]{lw} that $\frac{1}{2}(-tb_{\mathbb{Q}}(L)\pm rot_{\mathbb{Q}}(L)-1)\geq -\tau^{\ast}_{c(\xi)}(Y, K)$. Applying Corollary~\ref{cor:cone}, we conclude that $\nu_{c(\xi)}(-Y, K)=-\tau^{\ast}_{c(\xi)}(Y,K)+1$ and $-\tau^{\ast}_{c(\xi)}(Y,K)=A(\underline{\mathfrak{s}})=\frac{1}{2}(-tb_{\mathbb{Q}}(L)\pm rot_{\mathbb{Q}}(L)-1)$.
\end{proof}

\begin{proof}[Proof of Theorem~\ref{Theorem:Main1}]
Since $\frac{1}{2}(tb_{\mathbb{Q}}(L)\pm rot_{\mathbb{Q}}(L)+1)<\tau^{\ast}_{c(\xi)}(Y,K)=-\tau_{c(\xi)}(-Y,K)$,  Lemma~\ref{lemma: im} implies that $c(\xi)$ lies in the image of $v_{\underline{\mathfrak{s}}\ast}$. It suffices to find a cycle $c\in A_{\underline{\mathfrak{s}}}$ such that $v_{\underline{\mathfrak{s}}\ast}([c])=c(\xi)\in H_{\ast}(B_{\underline{\mathfrak{s}}})$, while $h_{\underline{\mathfrak{s}}\ast}([c])=0$. Recall that $A_{\underline{\mathfrak{s}}}$ is the subquotient complex  $C_{\underline{\mathfrak{s}}}\{\text{max}(i,j)=0\}$ of $CFK^{\infty}(-Y,K,\underline{\mathfrak{s}})$. By the definition of $\tau_{c(\xi)}(-Y,K)$, there is a cycle $c$ in the vertical complex $B_{\underline{\mathfrak{s}}}=C_{\underline{\mathfrak{s}}}\{i=0\}$ that is supported in $C_{\underline{\mathfrak{s}}}\{i=0, j\leq  \tau_{c(\xi)}(-Y,K)-A(\underline{\mathfrak{s}})\}$ and $[c]=c(\xi)\in H_{\ast}(B_{\underline{\mathfrak{s}}})$. By our assumption, $A(\underline{\mathfrak{s}})=\frac{1}{2}(\pm rot_{\mathbb{Q}}(L)-tb_{\mathbb{Q}}(L)-1)>\tau_{c(\xi)}(-Y,K)$, so $c$ can be considered as a cycle in $A_{\underline{\mathfrak{s}}}$, and since it lies in the subcomplex with $j<0$, it vanishes under $h_{\underline{\mathfrak{s}}}$.
\end{proof}

\begin{remark}
\begin{figure}[htb]
\begin{overpic}
{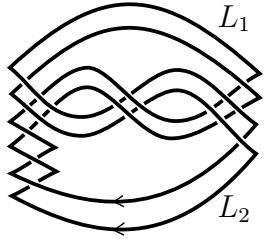}
\put(80, 8){$L_2$}
\put(80, 80){$L_{1}$}
\end{overpic}
\caption{A Legendrian link $L_{1}\cup L_{2}$, where $L_2$ is a Legendrian push-off of $L_1$.}
\label{figure:rot}
\end{figure}
For a Legendrian knot $L$ in $(S^3,\xi_{st})$, if $rot(L)\neq 0$, then \cite[Theorem 1.1]{gol} tells us that $\xi_{+1}(L)$ has vanishing contact invariant. However, this is not true for Legendrian knots in contact rational homology 3-spheres. For example, in Figure~\ref{figure:rot}, $L_1$ is a Legendrian right handed trefoil with $tb(L_1)=0$ and $rot(L_1)=-1$, and $L_2$ is a Legendrian push-off of $L_1$. The 3-manifold $S^3(L_{2}^{-})$ is an integral homology 3-sphere. By \cite[Lemma 3.1]{go},  the Thurston-Bennequin invariant of $L_1$ in  $(S^3(L_{2}^{-}),\xi_{st}(L_2^-))$ is $0$ and the rotation number of $L_1$ in $(S^3(L_{2}^{-}),\xi_{st}(L_2^-))$ is $-1$. Contact $(+1)$-surgery along $L_1$ in  $(S^3(L_{2}^{-}),\xi_{st}(L_2^-))$ yields $(S^3, \xi_{st})$ (see \cite{dg1}) which certainly has nonvanishing contact invariant.
\end{remark}

Now we turn to some applications of Theorem~\ref{Theorem:Main1}. First we recall the following proposition.
\begin{proposition} \label{prop: add} \cite[Lemma 3.2]{lw}
For $i=1,2$, suppose that $L_{i}$ is a Legendrian knot in a contact rational homology 3-sphere $(Y_{i}, \xi_{i})$.  Then the rational Thurston-Bennequin invariant and the rational rotation number of the Legendrian knot $L_{1}\sharp L_{2}$ in the contact 3-manifold $(Y_{1}\sharp Y_{2},  \xi_{1}\sharp\xi_{2})$  satisfy
$$tb_{\mathbb{Q}}(L_{1}\sharp L_{2})=tb_{\mathbb{Q}}(L_{1})+tb_{\mathbb{Q}}(L_{2})+1,$$
$$rot_{\mathbb{Q}}(L_{1}\sharp L_{2})=rot_{\mathbb{Q}}(L_{1})+rot_{\mathbb{Q}}(L_{2}).$$
\end{proposition}

\begin{proof}[Proof of Proposition~\ref{Pro0}]
It suffices to prove the case that $\mathbb{L}_1$ contains only one component $L_1$. Suppose $L_{1}$ is the Legendrian connected sum of $L'_{3}$ and $L_3$. Then we have  $$(S^3(\mathbb{L}_{2}^{-}), L_1)=(S^3(\mathbb{L}_{2}^{-}), L'_3)\sharp(S^3, L_3).$$
By \cite[Theorem 1.1]{lw},
$$tb_{\mathbb{Q}}(L'_3)+|rot_{\mathbb{Q}}(L'_3)|+1\leq 2\tau^{\ast}_{c(\xi_{st}(\mathbb{L}_{2}^{-}))}(S^3(\mathbb{L}_{2}^{-}), L'_3).$$
By assumption,
$$tb(L_3)+|rot(L_3)|+1<2\tau(L_3).$$
So by Propositions~\ref{prop: additivity} and \ref{prop: add}, we have
$$tb_{\mathbb{Q}}(L_1)+|rot_{\mathbb{Q}}(L_1)|+1< 2\tau^{\ast}_{c(\xi_{st}(\mathbb{L}_{2}^{-1}))}(S^3(\mathbb{L}_{2}^{-}), L_1).$$
The proposition now follows from Theorem~\ref{Theorem:Main1}.
\end{proof}

\begin{figure}[htb]
\begin{overpic}
{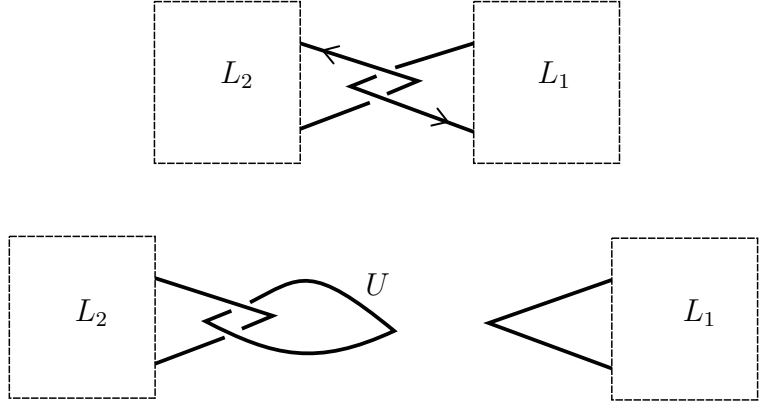}
\put(80, 120){$L_2$}
\put(200, 120){$L_{1}$}

\put(25, 30){$L_2$}
\put(135, 40){$U$}
\put(255, 30){$L_{1}$}
\end{overpic}
\caption{A Legendrian link $L_{1}\cup L_{2}$.}
\label{figure:lk1}
\end{figure}

\begin{corollary}\label{Cor1}
Let $L_{1}\cup L_{2}\subset (S^{3},\xi_{st})$ be an oriented Legendrian link with two components which has a front projection depicted at the top of Figure~\ref{figure:lk1}.  If $tb(L_2)\neq 1$ and  $tb(L_1)+|rot(L_1)|<2\tau(L_1)-1$, then the contact invariant $c(\xi_{st}(L_{1}^{+}\cup L_{2}^{-}))$ vanishes.
\end{corollary}

\begin{proof} Clearly, the contact invariant $c(\xi_{st}(L_{2}^{-}))$ is non-trivial. Since $tb(L_2)-1\neq 0$, $S^{3}(L_2^{-})$ is a rational homology 3-sphere. We have $$(S^{3}(L_2^{-}), L_1)=(S^{3}(L_2^{-}), U)\sharp (S^{3}, L_1),$$ where $U$ is a Legendrian unknot shown at the bottom left of Figure~\ref{figure:lk1}.

Since $$tb(L_1)+|rot(L_1)|+1<2\tau(L_1),$$
the corollary follows from Proposition~\ref{Pro0}.
\end{proof}

\begin{figure}[htb]
\begin{overpic}
{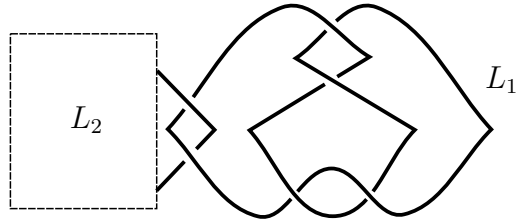}
\put(23, 35){$L_2$}
\put(180, 50){$L_{1}$}
\end{overpic}
\caption{A Legendrian link $L_{1}\cup L_{2}$, where $L_1$ is a Legendrian figure eight knot.}
\label{figure:lk3}
\end{figure}

\begin{example}\label{Example2}
Let $L=L_{1}\cup L_{2}$ be a Legendrian link in $(S^{3}, \xi_{st})$ depicted in Figure~\ref{figure:lk3}. Note that $L_{1}$ is a Legendrian figure eight knot with  $tb(L_{1})=-3$ and $rot(L_1)=0$. Since $\tau(L_1)=0$, Corollary~\ref{Cor1} implies that $(S^3(L_{1}^{+}\cup L_{2}^{-}),\xi_{st}(L_1^+\cup L_2^-))$ is a contact 3-manifold with vanishing contact invariant for any Legendrian knot $L_2$ with $tb(L_2)\neq 1$.
\end{example}

In the last part of this section, we prove Proposition~\ref{Prop:tb} and its application to contact $(+1)$-surgeries along Legendrian links in $(S^3,\xi_{st})$, Corollary~\ref{cor0}.  Note that the vanishing result in Proposition~\ref{Prop:tb} is only obtained for the plus-version of the contact invariant $c^+(\xi)$ as opposed to $c(\xi)$ in the other parts of the paper.
To the best of the authors' knowledge, there is no known example of contact 3-manifold $(Y,\xi)$ with vanishing $c^{+}(\xi)$ but nonvanishing $c(\xi)$.

\begin{proof}[Proof of Proposition~\ref{Prop:tb}]
Let $W$ be the cobordism from $Y$ to $Y_{+1}(L)$ induced by contact $(+1)$-surgery. Then the map $F_{-W}^{+}:HF^{+}(-Y)\rightarrow HF^{+}(-Y_{+1}(L))$ send $c^{+}(\xi)$ to $c^{+}(\xi_{+1}(L))$. By \cite[Lemma 5.1]{mt} and Lemma~\ref{lemma: tb}, the self-intersection of a generator of $H_{2}(-W;\mathbb{Z})\cong \mathbb{Z}$ is $-q^2(tb_{\mathbb{Q}}(L)+1)>0$. So $-W$ is positive definite. By \cite[Lemma 8.2]{OSzFour}, $F_{-W}^{\infty}: HF^{\infty}(-Y)\rightarrow HF^{\infty}(-Y_{+1}(L))$ vanishes. Since $Y$ is an L-space, $HF^{\infty}(-Y)\rightarrow HF^{+}(-Y)$ is onto. Hence $F_{-W}^{+}=0$, and the contact invariant $c^{+}(\xi_{+1}(L))$ vanishes.
\end{proof}

\begin{proof}[Proof of Corollary~\ref{cor0}] If $L_2$ is an unknot, then the corollary follows from \cite[Theorem 1.1]{dlw}. In the following we assume that $L_2$ is nontrivial.

Since $L_2$ is an L-space knot, $g(L_2)=\tau(L_2)$.  If $tb(L_2)<2g(L_2)-1=2\tau(L_2)-1$, then \cite[Theorem 1.1]{gol} implies that $c^{+}(\xi_{st}(L^{+}_{2}))$ vanishes. So the contact invariant  $c^{+}(\xi_{st}(\mathbb{L}^{+}))$ vanishes for any Legendrian knot $L_1$.

From Bennequin inequality, we can now assume that $tb(L_2)=2g(L_2)-1$. (Indeed, Lidman and Sivek conjectured in \cite[Conjecture 1.19]{ls2} that any L-space knot $K$ has maximal Thurston-Bennequin invariant $2g(K)-1$.) By \cite[Theorem 1.1]{gol},  $(S^3(L_{2}^{+}),\xi_{st}(L_2^+))$ is a tight contact L-space.  Using \cite[Lemma 3.1]{go}, we know that the rational Thurston-Bennequin invariant of $L_1$ in $(S^3(L_{2}^{+}),\xi_{st}(L_2^+))$ is  $$tb(L_1)+\frac{\det\left(
\begin{array}{cc}
0 & l \\
l & tb(L_2)+1
\end{array}
\right)}{tb(L_2)+1}=tb(L_1)-\frac{l^2}{2g(L_2)},$$ which is less than $-1$ by assumption. So the corollary follows immediately  from Proposition~\ref{Prop:tb}.
\end{proof}

\begin{example}\label{Example11}
Contact  $(+1)$-surgery along the Legendrian link  $\mathbb{L}=L_{1}\cup L_{2}$ in $(S^3,\xi_{st})$ depicted in Figure~\ref{figure: lk11} yields a contact 3-manifold with vanishing contact invariant $c^{+}(\xi_{st}(\mathbb{L}^+))$.
\end{example}

\begin{figure}[htb]
\begin{overpic}
{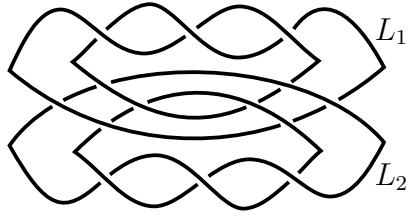}
\put(140, 65){$L_1$}
\put(140, 10){$L_2$}
\end{overpic}
\caption{A Legendrian link $L_{1}\cup L_{2}$, where both components are Legendrian right handed trefoil knots with $tb=1$, and the linking number is $4$.}
\label{figure: lk11}
\end{figure}

\section{Contact ($\pm$1)-surgeries on rational homology 3-spheres}

Let $\mathbb{L}=L_1\cup\cdots\cup L_n$ be a Legendrian link in a contact rational homology 3-sphere $(Y,\xi)$. Suppose that contact $(\pm 1)$-surgery along the components of $\mathbb{L}$ produces $(Y',\xi')$. Set $a_i=tb_{\mathbb{Q}}(L_i)\pm 1$, depending on whether we perform contact $(+1)$-surgery or $(-1)$-surgery along $L_i$. Let $L_0$ be a Legendrian knot in $(Y,\xi)$ disjoint from $\mathbb{L}$. For $i\neq j$, denote the rational linking number $lk_{\mathbb{Q}}(L_i,L_j)$ by $l_{ij}$. Let $M=(m_{ij})_{i,j=1}^n$ be the linking matrix of $\mathbb{L}$, i.e.
$$m_{ij}=\left\{ \begin{array}{ll} a_i & \mathrm{if}\ i=j, \\ l_{ij} & \mathrm{if}\ i\neq j.\end{array}\right.$$
Let $M_0=(m_{ij})_{i,j=0}^n$ be the extended matrix given by
$$m_{ij}=\left\{ \begin{array}{ll} 0 & \mathrm{if}\ i=j=0, \\ a_i & \mathrm{if}\ i=j\ge 1, \\ l_{ij} & \mathrm{if}\ i\neq j. \end{array}\right.$$
$Y'$ is still a rational homology 3-sphere if and only if $\det M\neq 0$ (see the proof of the following lemma). Let $L$ be the image of $L_0$ in $(Y',\xi')$.
The following lemma is a generalization of \cite[Lemma 3.1]{go}.

\begin{lemma} \label{tb and rot} Suppose $\det M\neq 0$. Then
$$tb_{\mathbb{Q}}(L)=tb_{\mathbb{Q}}(L_0)+\frac{\det M_0}{\det M}$$
and
$$rot_{\mathbb{Q}}(L)=rot_{\mathbb{Q}}(L_0)-\left< \left( \begin{array}{c} rot_{\mathbb{Q}}(L_1) \\ \vdots \\ rot_{\mathbb{Q}}(L_n) \end{array} \right) ,\
M^{-1} \left( \begin{array}{c} l_{01} \\ \vdots \\ l_{0n} \end{array} \right) \right> .$$
\end{lemma}

\begin{proof}[Proof]
For $i=0,1,\ldots ,n$, denote by $N(L_i)$ a closed tubular neighborhood of $L_i$ in $Y$. Suppose for $i\neq j$, $N(L_i)$ and $N(L_j)$ are disjoint. Denote the knot exterior
$\overline{Y\setminus N(L_i)}$ by $X(L_i)$. Denote $\overline{Y\setminus\bigcup\limits_{i=1}^n N(L_i)}$ by $X(\mathbb{L})$ and $\overline{Y\setminus\bigcup\limits_{i=0}^n N(L_i)}$
by $X(L_0\cup\mathbb{L})$. Suppose the order of $[L_i]$ in $H_1(Y;\mathbb{Z})$ is $q_i$. Let $F_i$ be a rational Seifert surface for $L_i$. We can assume
that $F_i\cap\partial N(L_i)$ is composed of $c_i$ parallel oriented simple closed curves, each of which has homology $t_i\lambda_i+r_i\mu_i\in H_1(\partial N(L_i);\mathbb{Z})$, where $\lambda_i$ is the class of a canonical longitude and $\mu_i$ is the class of a meridian, $t_i,r_i$ are coprime and $0\le r_i<t_i$. Certainly we have $c_it_i=q_i$.

Note that
$$H_1(X(L_0\cup\mathbb{L});\mathbb{Q})\cong \mathbb{Q}\langle\mu_0\rangle\oplus\cdots\oplus\mathbb{Q}\langle\mu_n\rangle,$$
where $\mathbb{Q}\langle\mu_i\rangle$ denotes the vector space over $\mathbb{Q}$ generated by the class $\mu_i$. In $H_1(X(L_0\cup\mathbb{L});\mathbb{Q})$, we have
$$c_i(t_i\lambda_i+r_i\mu_i)=\sum_{j=0 \atop j\neq i}^{n} q_il_{ij}\mu_j.$$
Since $c_it_i=q_i$, this is equivalent to
$$\lambda_i=-\frac{r_i}{t_i}\mu_i+\sum_{j=0 \atop j\neq i}^{n} l_{ij}\mu_j.$$
Suppose the class of a longitude of $L_i$ determined by the contact framing is $\lambda_i+m_i\mu_i$. Then
$$tb_{\mathbb{Q}}(L_i)=\frac{1}{q_i}[F_i]\bullet (\lambda_i+m_i\mu_i)=\frac{1}{q_i}(c_it_im_i-c_ir_i)=m_i-\frac{r_i}{t_i}.$$
Contact $(\pm 1)$-surgery along $L_i$ implies that we glue a meridional disc along $$\lambda_i+(m_i\pm 1)\mu_i=a_i\mu_i+\sum_{j=0 \atop j\neq i}^{n} l_{ij}\mu_j,$$
$i=1,\ldots,n$. It follows that
$$H_1(X'(L);\mathbb{Q})\cong \mathbb{Q}\langle\mu_0\rangle\oplus\cdots\oplus\mathbb{Q}\langle\mu_n\rangle/\langle a_i\mu_i+\sum_{j=0 \atop j\neq i}^{n} l_{ij}\mu_j,\ i=1,\ldots,n\rangle,$$
where $X'(L)$ denotes the knot exterior $\overline{Y'\setminus N(L_0)}$.
Similarly,
$$H_1(Y';\mathbb{Q})\cong \mathbb{Q}\langle\mu_1\rangle\oplus\cdots\oplus\mathbb{Q}\langle\mu_n\rangle/\langle a_i\mu_i+\sum_{j=1 \atop j\neq i}^{n} l_{ij}\mu_j,\ i=1,\ldots,n\rangle.$$
Hence $Y'$ is a rational homology 3-sphere if and only if $\det M\neq 0$.

Since $\det M\neq 0$, $Y'$ is a rational homology 3-sphere and $H_1(X'(L);\mathbb{Q})\cong\mathbb{Q}\langle \mu_0\rangle$. Thus there is a unique $a_0\in\mathbb{Q}$ such that $\lambda_0+a_0\mu_0=0$ in $H_1(X'(L);\mathbb{Q})$. The rational
Thurston-Bennequin invariant of $L$ in $(Y',\xi')$ can be computed as
$$tb_{\mathbb{Q}}(L)=(\lambda_0+a_0\mu_0)\bullet (\lambda_0+m_0\mu_0)=m_0-a_0,$$
where the intersection product is on $\partial X'(L)$. Since
$$\lambda_0=-\frac{r_0}{t_0}\mu_0+\sum_{j=1}^{n} l_{0j}\mu_j\in H_1(X(L_0\cup\mathbb{L});\mathbb{Q}),$$
and $$\lambda_0+a_0\mu_0=0\in H_1(X'(L);\mathbb{Q}),$$
$(a_0-\frac{r_0}{t_0})\mu_0+\sum\limits_{j=1}^{n} l_{0j}\mu_j$ is a linear combination of the relations in
$H_1(X'(L);\mathbb{Q})$, which gives
$$0=\left| \begin{array}{cccc} a_0-\frac{r_0}{t_0} & l_{01} & \cdots & l_{0n} \\ l_{10} & a_1 & \cdots & l_{1n} \\ \vdots & \vdots
& \ddots & \vdots \\ l_{n0} & l_{n1} & \cdots & a_n \end{array} \right| =(a_0-\frac{r_0}{t_0})\det M+\det M_0.$$
Hence $tb_{\mathbb{Q}}(L)=m_0-(\frac{r_0}{t_0}-\frac{\det M_0}{\det M})=tb_{\mathbb{Q}}(L_0)+\frac{\det M_0}{\det M}$.

The Poincar\'{e} dual of the relative class $e(\xi, L_i)$ over $\mathbb{Q}$, $\mathrm{PD}(e(\xi,L_i))_{\mathbb{Q}}\in H_1(X(L_i);\mathbb{Q})\cong\mathbb{Q}\langle\mu_i\rangle$, is
equal to $rot_{\mathbb{Q}}(L_i)\mu_i$. Since under the map $H_1(X(L_0\cup\mathbb{L});\mathbb{Q})\to H_1(X(L_i);\mathbb{Q})$ induced by inclusion, $\mathrm{PD}(e(\xi,
\bigcup\limits_{i=0}^n L_i))_\mathbb{Q}$ maps to $\mathrm{PD}(e(\xi, L_i))_{\mathbb{Q}}$, we have
$$\mathrm{PD}(e(\xi,\bigcup^n_{i=0} L_i))_\mathbb{Q}=\sum_{i=0}^n rot_{\mathbb{Q}}(L_i)\mu_i.$$
Under the map $H_1(X(L_0\cup\mathbb{L});\mathbb{Q})\to H_1(X'(L);\mathbb{Q})$ induced by inclusion, $\mathrm{PD}(e(\xi,
\bigcup\limits_{i=0}^n L_i))_\mathbb{Q}$ maps to $\mathrm{PD}(e(\xi', L))_{\mathbb{Q}}$ (see \cite[Lemma 3.2]{go}). Therefore, in $H_1(X'(L);\mathbb{Q})$, we have
$$rot_{\mathbb{Q}}(L)\mu_0=\sum_{i=0}^n rot_{\mathbb{Q}}(L_i)\mu_i$$
$$=\left( rot_{\mathbb{Q}}(L_0)-\left< \left( \begin{array}{c} rot_{\mathbb{Q}}(L_1) \\ \vdots \\ rot_{\mathbb{Q}}(L_n) \end{array} \right) ,\
M^{-1} \left( \begin{array}{c} l_{01} \\ \vdots \\ l_{0n} \end{array} \right) \right> \right)\mu_0.$$
This proves the second formula in the lemma.
\end{proof}

\section{Overtwisted contact surgeries on rational homology 3-spheres}

We use the following result as the main tool of the proof of Theorem~\ref{Theorem:Main3}.

\begin{theorem}[\'{S}wi\c{a}towski\cite{d}, Etnyre\cite{e1}, Baker-Onaran\cite{bo}] \label{loose}
If $L\subset (Y,\xi)$ is a rationally null-homologous Legendrian knot such that  the complement of a regular neighborhood of $L$ is tight, then $$-|tb_{\mathbb{Q}}(L)|+|rot_{\mathbb{Q}}(L)|\leq -\frac{\chi(F)}{q},$$ where $q$ is the order of $[L]$ in $H_{1}(Y;\mathbb{Z})$, and $\chi(F)$ is the Euler characteristic of a rational Seifert surface $F$ for $L$.
\end{theorem}

\begin{proof}[Proof of Theorem~\ref{Theorem:Main3}]

Consider contact $(+1)$-surgery along $L$. Let $L^*$ be the surgery dual. Let $L_0$ be a Legendrian push-off of $L$. Then $L_0^*$, the image of $L_0$ in
$(Y_{+1}(L),\xi_{+1}(L))$, is Legendrian isotopic to $L^*$. We use Lemma~\ref{tb and rot} to compute $tb_{\mathbb{Q}}(L_0^*)$ and $rot_{\mathbb{Q}}(L_0^*)$.
Now $M=(tb_{\mathbb{Q}}(L)+1)$ and $M_0=\left( \begin{array}{cc} 0 & tb_{\mathbb{Q}}(L) \\ tb_{\mathbb{Q}}(L) & tb_{\mathbb{Q}}(L)+1 \end{array} \right)$.
Hence
$$tb_{\mathbb{Q}}(L_0^*) =tb_{\mathbb{Q}}(L_0)+\frac{\det M_0}{\det M} $$
$$=tb_{\mathbb{Q}}(L)-\frac{(tb_{\mathbb{Q}}(L))^2}{tb_{\mathbb{Q}}(L)+1} =\frac{tb_{\mathbb{Q}}(L)}{tb_{\mathbb{Q}}(L)+1}$$
and
$$rot_{\mathbb{Q}}(L_0^*)=rot_{\mathbb{Q}}(L_0)-\frac{rot_{\mathbb{Q}}(L)\cdot tb_{\mathbb{Q}}(L)}{tb_{\mathbb{Q}}(L)+1}=\frac{rot_{\mathbb{Q}}(L)}{tb_{\mathbb{Q}}(L)+1}.$$

\begin{lemma}\label{lemma:order3}
The order of $[L_{0}^{\ast}]$ in $H_{1}(Y_{+1}(L);\mathbb{Z})$ is $q\cdot|tb_{\mathbb{Q}}(L)+1|$.
\end{lemma}

\begin{proof}
We use the notation in the first paragraph of Section 2 with $K$ replaced by $L$. Let $X(L)$ denote the knot exterior $\overline{Y\setminus N(L)}$. Let $i_*:H_1(\partial N(L);\mathbb{Z})\to H_1(X(L);\mathbb{Z})$ be the map induced by inclusion. Then $\ker i_*$ is generated by $c(t\lambda_{can}+r\mu)=q\lambda_{can}+cr\mu$.
Write $\mu'$ and $\lambda'$ for the classes of a meridian and a longitude, respectively, of the solid torus we glue in to perform the surgery.
Suppose the class of a longitude of $L$ determined by the contact framing is $\lambda_c=\lambda_{can}+(p-1)\mu$ for some integer $p$. Contact $(+1)$-surgery can be described
by the map
$$\mu'\mapsto \mu+\lambda_c=p\mu+\lambda_{can},\ \lambda'\mapsto \lambda_c=(p-1)\mu+\lambda_{can}.$$
Then in terms of $\mu'$ and $\lambda'$, $\ker i_*$ is generated by $(pq-cr)\lambda'+(cr-pq+q)\mu'$. Thus the order of $[L^*]$ in $H_1(Y_{+1}(L);\mathbb{Z})$ is $|pq-cr|$.
Since $L_0^*$ is isotopic to $L^*$, by Lemma~\ref{lemma: tb}, $[L_{0}^{\ast}]$ is of order $q\cdot |tb_{\mathbb{Q}}(L)+1|$ in $H_{1}(Y_{+1}(L);\mathbb{Z})$.
\end{proof}

(1) Let $L_{k}^{\ast}$ be $k$ positive or negative stabilization of $L_{0}^{\ast}$. If $rot_{\mathbb{Q}}(L_{0}^{\ast})$ is nonnegative, then we choose $k$ positive stabilization. Otherwise, we choose  $k$ negative stabilization. Assume that $k$ is sufficiently large. Since $tb_{\mathbb{Q}}(L)<-1$, $tb_{\mathbb{Q}}(L_{0}^{\ast})$ is positive. So we have $$-|tb_{\mathbb{Q}}(L_{k}^{\ast})|+|rot_{\mathbb{Q}}(L_{k}^{\ast})|=-|tb_{\mathbb{Q}}(L_{0}^{\ast})-k|+|rot_{\mathbb{Q}}(L_{0}^{\ast})|+k$$
$$=-(k-tb_{\mathbb{Q}}(L_{0}^{\ast}))+|rot_{\mathbb{Q}}(L_{0}^{\ast})|+k
=|tb_{\mathbb{Q}}(L_{0}^{\ast})|+|rot_{\mathbb{Q}}(L_{0}^{\ast})|$$
$$=\frac{|tb_{\mathbb{Q}}(L)|+|rot_{\mathbb{Q}}(L)|}{|tb_{\mathbb{Q}}(L)+1|}.$$

The order of $[L_{k}^{\ast}]$ in $H_{1}(Y_{+1}(L);\mathbb{Z})$ is the same as that of $[L_0^*]$, that is, $q|tb_{\mathbb{Q}}(L)+1|$.
Denote $F\cap X(L)$ by $F^0$. We can radially cone $\partial F^0\subset\partial X(L)$ in the solid torus we glue in to get a rational Seifert surface $F^*$ for $L^*$ in $Y_{+1}(L)$. Since $L_{k}^{\ast}$ is smoothly isotopic to $L^*$, there is a rational Seifert surface $F_k^*$ for $L_k^*$ in $Y_{+1}(L)$ with $\chi (F_k^*)=\chi (F^*)=\chi(F)$. Since $tb_{\mathbb{Q}}(L)-|rot_{\mathbb{Q}}(L)|<\frac{\chi(F)}{q}$, $|tb_{\mathbb{Q}}(L)|+|rot_{\mathbb{Q}}(L)|>\frac{-\chi(F)}{q}$ and $-|tb_{\mathbb{Q}}(L_{k}^{\ast})|+|rot_{\mathbb{Q}}(L_{k}^{\ast})|>\frac{-\chi(F_k^*)}{q|tb_{\mathbb{Q}}(L)+1|}$. By Theorem~\ref{loose}, the complement of $L_{k}^{\ast}$ in $(Y_{+1}(L), \xi_{+1}(L))$ is overtwisted.

(2) Since $\chi(F)\le 1$, $tb_{\mathbb{Q}}(L)+|rot_{\mathbb{Q}}(L)|<\frac{\chi(F)}{q}-2$ implies that $tb_{\mathbb{Q}}(L)< -1$.
Consider $L_+^*$ and $L_-^*$, the positive and negative stabilizations of $L_0^*$. We have
$$tb_{\mathbb{Q}}(L_+^*)=tb_{\mathbb{Q}}(L_0^*)-1=-\frac{1}{tb_{\mathbb{Q}}(L)+1}$$ and
$$rot_{\mathbb{Q}}(L_+^*)=rot_{\mathbb{Q}}(L_0^*)+1=\frac{rot_{\mathbb{Q}}(L)+tb_{\mathbb{Q}}(L)+1}{tb_{\mathbb{Q}}(L)+1}.$$
It follows that
$$-|tb_{\mathbb{Q}}(L_{+}^{\ast})|+|rot_{\mathbb{Q}}(L_{+}^{\ast})|=\frac{|rot_{\mathbb{Q}}(L)+tb_{\mathbb{Q}}(L)+1|-1}{|tb_{\mathbb{Q}}(L)+1|}>-\frac{\chi (F)}{q|tb_{\mathbb{Q}}(L)+1|}.$$
So $L_+^*$ is loose by Theorem~\ref{loose}. Similarly, $L_-^*$ is also loose.

Let $X_+^*(L)$ (respectively, $X_-^*(L)$) denote the complement of a standard neighbourhood of $L_+^*$ (respectively, $L_-^*$) in $(Y_{+1}(L),\xi_{+1}(L))$.
Then $(X_+^*(L),\xi_{+1}(L))$ and $(X_-^*(L),\xi_{+1}(L))$ are overtwisted. Since the result of any positive contact surgery along $L$ in $(Y,\xi)$ contains either
$(X_+^*(L),\xi_{+1}(L))$ or $(X_-^*(L),\xi_{+1}(L))$ (see \cite[Section 4]{c}), it is overtwisted.
\end{proof}

We give some applications of Theorem~\ref{Theorem:Main3}. In practice, the most difficult part is to find a rational Seifert surface.

\begin{corollary}\label{Cor2}
Let $L_{1}\cup L_{2}\subset (S^{3},\xi_{st})$ be an oriented Legendrian link with two components which has a front projection depicted at the top of Figure~\ref{figure:lk1}.  If $$tb(L_2)\neq 1,\, tb(L_1)+\frac{1}{1-tb(L_2)}<-1$$ and $$|rot(L_1)+\frac{rot(L_2)}{1-tb(L_2)}|>2g_{1}+\frac{2g_{2}-1}{|1-tb(L_2)|}+tb(L_1)+\frac{1}{1-tb(L_2)},$$ where $g_i$ is the genus of $L_i$ for $i=1,2$, then  $(S^3(L_{1}^{+}\cup L_{2}^{-}),\xi_{st}(L_1^+\cup L_2^-))$ is overtwisted.
\end{corollary}

\begin{proof}
Let $L$ be the image of $L_1$ in $(S^3(L_2^{-}),\xi_{st}(L_2^-))$. By Lemma~\ref{tb and rot}, $tb_{\mathbb{Q}}(L)=tb(L_1)+\frac{1}{1-tb(L_2)}$ and $rot_{\mathbb{Q}}(L)=rot(L_1)+\frac{rot(L_2)}{1-tb(L_2)}$. The order $q$ of $[L]$ in $H_1(S^3(L_{2}^{-});\mathbb{Z})$ is $|1-tb(L_2)|$. The Legendrian knot  $L$ in
$(S^3(L_{2}^{-}),\xi_{st}(L_2^-))$ can be seen as the connected sum of a Legendrian knot $U$ in $(S^3(L_{2}^{-}),\xi_{st}(L_2^-))$ and $L_{1}$ in $(S^3,\xi_{st})$ (see the bottom of Figure~\ref{figure:lk1}). The order of $[U]$ in $H_1(S^3(L_{2}^{-});\mathbb{Z})$ is also $|1-tb(L_2)|$. Since $U$ is smoothly isotopic to the core of the solid torus we glue in to get $S^3(L_{2}^{-})$, it has a rational Seifert surface in $S^3(L_2^-)$ with Euler characteristic $1-2g_2$. By \cite[(2.3.1)]{cg}, $L$ has a rational Seifert surface $F$ in $S^3(L_{2}^{-})$ with Euler characteristic $$1-2g_2+|1-tb(L_2)|\cdot(1-2g_1)-|1-tb(L_2)|=1-2g_{2}-2g_1\cdot|1-tb(L_2)|.$$
So we have $tb_{\mathbb{Q}}(L)<-1$ and $tb_{\mathbb{Q}}(L)-|rot_{\mathbb{Q}}(L)|<\frac{\chi(F)}{q}$. By Theorem~\ref{Theorem:Main3}, $(S^3(L_{1}^{+}\cup L_{2}^{-}),\xi_{st}(L_1^+\cup L_2^-))$ is overtwisted.
\end{proof}

\begin{figure}[htb]
\begin{overpic}
{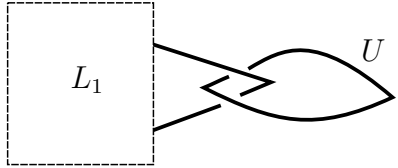}

\put(25, 30){$L_1$}
\put(135, 40){$U$}
\end{overpic}
\caption{A Legendrian link $L_{1}\cup U$.}
\label{figure:lk4}
\end{figure}

\begin{example}\label{Example3}
Let $L_{1}\cup U$ be the Legendrian link in $(S^3,\xi_{st})$ shown in Figure~\ref{figure:lk4}. By Corollary~\ref{Cor2}, if $tb(L_1)\leq -2$ and $|rot(L_1)|>tb(L_1)+2g_1$, where $g_1$ is the genus of $L_1$, then $(S^3(L_{1}^{+}\cup U^{-}),\xi_{st}(L_1^+\cup U^-))$ is overtwisted.
\end{example}

\begin{example}\label{Example31}
Let $L_{1}\cup U$ be the Legendrian link in $(S^3,\xi_{st})$ shown in Figure~\ref{figure:lk41}. Denote the image of $L_1$ in $(S^3(U^{-}),\xi_{st}(U^-))$ by $L$. Obviously, $L$ is  null-homologous in $S^3(U^{-})$. Moreover, $tb_{\mathbb{Q}}(L)=-6$, $|rot_{\mathbb{Q}}(L)|=1$ and there exists a Seifert surface $F$ with $\chi(F)=-3$ for $L$ in $S^3(U^-)$ by tubing operation. By Theorem~\ref{Theorem:Main3} or \cite[Theorem 1.1]{c},  $(S^3(L_{1}^{+}\cup U^{-}),\xi_{st}(L_1^+\cup U^-))$ is overtwisted.
\end{example}

\begin{figure}[htb]
\begin{overpic}
{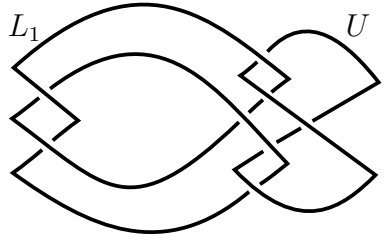}
\put(0, 75){$L_1$}
\put(128, 75){$U$}
\end{overpic}
\caption{A Legendrian link $L_{1}\cup U$ with linking number $0$.}
\label{figure:lk41}
\end{figure}

\begin{example}\label{Example4}
In \cite[Proposition 3.4]{bg}, Baker and Grigsby proved that any Legendrian knot $L$ in a universally tight contact lens space $(L(a,b),\xi_{UT})$ has a twisted toroidal front projection.  The invariant $tb_{\mathbb{Q}}(L)$ can be computed via such a front projection (see \cite[Proposition 6.8]{bg} and \cite[Corollary 3.3]{c1}). So we can apply Proposition~\ref{Prop:tb} for Legendrian knots in  $(L(a,b),\xi_{UT})$ conveniently.  In \cite[Proposition 3.6]{c1}, the invariant $rot_{\mathbb{Q}}(L)$ is also computed via the front projection. Possibly, one can construct a rational Seifert surface for $L$ via the front projection in a similar way as in \cite[Section 3.4]{oss}. Then we can apply Theorem~\ref{Theorem:Main3} for Legendrian knots in  $(L(a,b),\xi_{UT})$.
\end{example}

\section{Overtwisted contact (+1)-surgeries along Legendrian two-component links}

In this section, we prove Theorem~\ref{Theorem:Main2}.

\begin{proof}[Proof of Theorem~\ref{Theorem:Main2}]
We shall construct an explicit overtwisted disk in the contact 3-manifold $(S^3(\mathbb{L}^{+}),\xi_{st}(\mathbb{L}^{+}))$.
First, we construct a Legendrian knot $L'$ in $(S^3,\xi_{st})$ disjoint from the link $\mathbb{L}=L_{1}\cup L_{2}$. See Figure~\ref{figure:ot2}.
Outside the dashed box, $L'$ consists of two Legendrian arcs which are downward Legendrian push-offs of the parts of $L_{1}$ and $L_{2}$ outside the dashed box, respectively.
There is a thrice-punctured sphere $S$ shown in Figure~\ref{figure:ot2} whose boundary $\partial S=L_{1}\cup L_{2}\cup L'$.
\begin{figure}[htb]
\begin{overpic}
{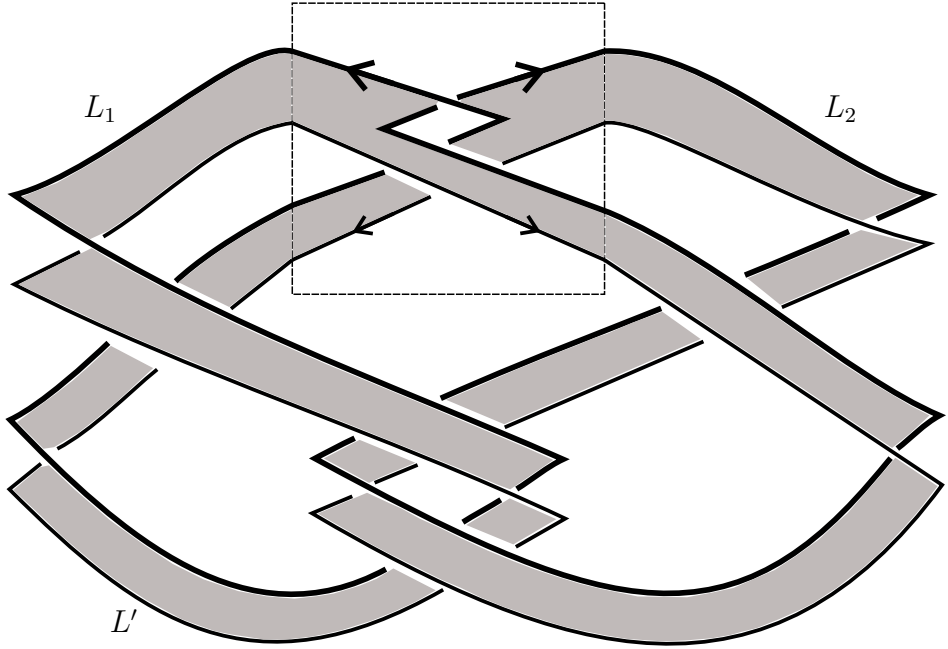}
\put(30, 200){$L_1$}
\put(310, 200){$L_2$}
\put(40, 5){$L'$}
\end{overpic}
\caption{The thin knot is $L'$. The shaded area is a thrice-punctured sphere. }
\label{figure:ot2}
\end{figure}

We orient $L_1,L_2$ and $L'$ as the boundary of $S$. Suppose the linking number of $L_1$ and $L_2$ is $l$.

\begin{lemma}  \label{lemma:tb(L')}
$tb(L')=tb(L_1)+tb(L_2)+2(l+1)$.
\end{lemma}
\begin{proof} The proof is similar to that of \cite[Lemma 6.1]{dlw}.  The number of cusps of $L'$ is $c(L')=c(L_1)+c(L_2)-2$. The writhe of $L'$ is $w(L')=w(L_1)+w(L_2)+2(l+1)-1$, where the self-crossings of $L'$ outside the dashed box contribute $w(L_1)+w(L_2)+2(l+1)$ to $w(L')$, and the self-crossing of $L'$ inside the dashed box contributes $-1$ to $w(L')$.  So $tb(L')=w(L')-\frac{1}{2}c(L')=tb(L_1)+tb(L_2)+2(l+1)$.
\end{proof}

\begin{lemma} \label{lemma:SurfaceFraming} (1) For $i=1,2$, the framing of $L_{i}$ induced by $S$ is $tb(L_i)+1$ with respect to the Seifert surface framing of $L_i$. \\ (2) The framing of $L'$  induced by $S$  is $tb(L_1)+tb(L_2)+2(l+1)$ with respect to the Seifert surface framing of $L'$; that is, the framing of $L'$ induced by $S$ coincides with the contact framing of $L'$.
\end{lemma}

\begin{proof} (1) For $i=1,2$, the framing of $L_{i}$ induced by $S$, with respect to the Seifert surface framing of $L_i$, is the linking number of $L_i$ and its push-off in the
interior of $S$. The verification is straightforward.\\
(2) Let $L'_0$ be the push-off of $L'$ in the interior of $S$. We compute the linking number $lk(L', L'_0)$ as the number of crossings where $L'_0$ crosses under $L'$, counted with sign.  The crossings outside the dashed box contribute $tb(L_1)+tb(L_2)+1+2(l+1)$ to $lk(L',L'_0)$. The crossing inside the dashed box contributes $-1$ to $lk(L',L'_0)$. So  $lk(L', L'_0)=tb(L_1)+tb(L_2)+2(l+1)$.
\end{proof}

By Lemma~\ref{lemma:SurfaceFraming}(1), after we perform contact $(+1)$-surgery along $L_{1}\cup L_{2}$, $S$ caps off to a disk with boundary $L'$. It follows from Lemma~\ref{lemma:SurfaceFraming}(2) that this disk is an overtwisted disk.
\end{proof}

Roger Casals provided an alternative proof of Theorem~\ref{Theorem:Main2}.

\begin{proof}[Alternative proof of Theorem~\ref{Theorem:Main2}]

\begin{figure}[htb]
\begin{overpic}
{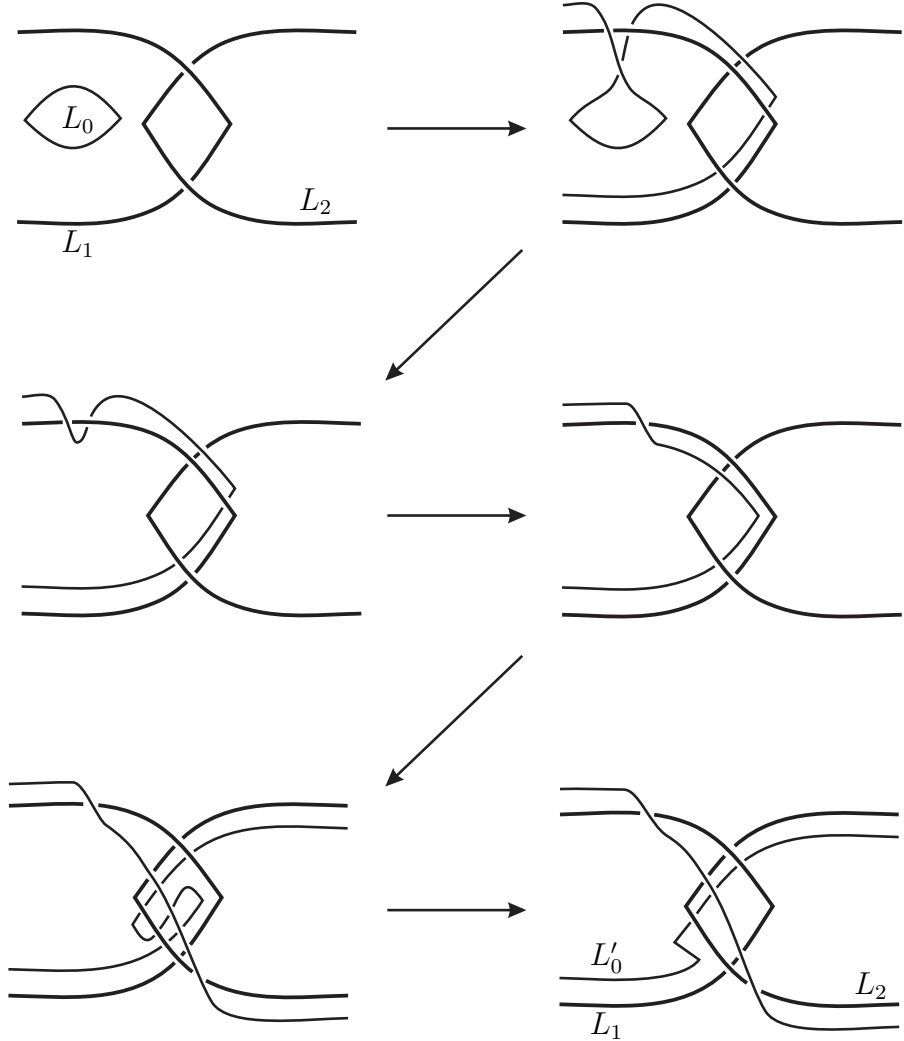}
\put(20, 342){$L_0$}
\put(20, 295){$L_1$}
\put(110, 311){$L_2$}
\put(220, -1){$L_1$}
\put(220, 25){$L_0'$}
\put(320, 14){$L_2$}
\end{overpic}
\caption{The thick curves present the two components of the link $\mathbb{L}=L_{1}\cup L_{2}$. The thin curves present $L_0$ and the resulting Legendrian knots after Legendrian Reidemeister moves and Kirby moves. The rest parts of the pictures are identical.}
\label{figure:ot3}
\end{figure}

We prove that the Legendrian unknot $L_0$ with $tb(L_0)=-1$ in $(S^3(\mathbb{L}^+),\xi_{st}(\mathbb{L}^+))$ in the first picture in the sequence of Figure~\ref{figure:ot3} destabilizes.
In fact, $L_0$ is Legendrian isotopic to the Legendrian knot $L_0'$ in $(S^3(\mathbb{L}^+),\xi_{st}(\mathbb{L}^+))$ in the final picture which contains an isolated stabilized arc.

In the first and fourth steps in the sequence of Figure~\ref{figure:ot3}, we perform Kirby moves of the second kind (see \cite[Proposition 1]{dg2}). The remaining moves are
Legendrian Reidemeister moves.
\end{proof}

\end{document}